\documentclass[12pt]{article}
\usepackage{amssymb}
\usepackage{mathrsfs}
\usepackage{amsfonts}
\usepackage{graphicx}
\usepackage{indentfirst}
\usepackage{colortbl,dcolumn}
\usepackage{color}
\usepackage{comment}
\usepackage{xcolor}
\usepackage{amsmath}
\usepackage{bbm} 
\usepackage{dsfont} 
\usepackage{enumerate}
\usepackage{psfrag}
\usepackage{booktabs}
\usepackage{array}
\usepackage{cite}
\usepackage{amsthm}
\numberwithin{equation}{section}
\usepackage[top=1in, bottom=1in, left=0.9in, right=0.9in]{geometry}
\usepackage{authblk}
\usepackage{caption}
\newcommand{\R}{\mathbb{R}}
\newcommand{\N}{\mathbb{N}}
\newcommand{\E}{\mathbb{E}}
\renewcommand{\P}{\mathscr{P}}

\newcommand{\F}{\mathcal{F}}

\newcommand{\1}{\mathbbm{1}}

\newtheorem{thm}{Theorem}[section]

\newtheorem{lem}[thm]{Lemma}

\newtheorem{example}[thm]{Example}
\newtheorem{assumption}[thm]{Assumption}

\begin{document}
\title{
Unconditionally positivity-preserving approximations of the A\"{i}t-Sahalia type model: Explicit Milstein-type schemes
\thanks{
This work was supported by Natural Science Foundation of China (12071488, 12371417, 11971488) and all authors contributed equally to this work.}
}
\author[a]{Yingsong Jiang\protect\footnotemark[4]}
\author[a]{Ruishu Liu\thanks{Corresponding author:
chicago@mail.ustc.edu.cn
}}
\author[a]{Xiaojie Wang\protect\footnotemark[4]}
\author[b]{Jinghua Zhuo\protect\footnotemark[4]}

\affil[a]{School of Mathematics and Statistics, HNP-LAMA, Central South University, Changsha, Hunan, China}
\affil[b]{Department of Economics, University of Warwick, Coventry CV4 7AL, United Kingdom}



\date{}
\maketitle	

\footnotetext[4]{Contributing author(s):
\begin{minipage}[t]{\linewidth}
    yingsong@csu.edu.cn
    \\
    x.j.wang7@csu.edu.cn
    \\
    Jinghua.Zhuo@warwick.ac.uk
\end{minipage}
}

\begin{abstract}

The present article aims to design and analyze efficient first-order strong schemes for a generalized A\"{i}t-Sahalia type model arising in mathematical finance and evolving in a positive domain $(0, \infty)$, which possesses a diffusion term with superlinear growth and a highly nonlinear drift that blows up at the origin.
Such a complicated structure of the model unavoidably causes essential difficulties in the construction and convergence analysis of time discretizations.
%
By incorporating implicitness in the term 
$\alpha_{-1} x^{-1}$ and a corrective mapping 
$\Phi_h$ in the recursion, 
we develop a novel class of explicit and unconditionally positivity-preserving (i.e., for any step-size $h>0$)  Milstein-type schemes for the underlying model. 
In both non-critical and general critical cases, 
we introduce a novel approach to analyze mean-square error bounds of the novel schemes,
without relying on a priori high-order moment bounds of the numerical approximations.
The expected order-one mean-square convergence is attained for the proposed scheme.
The above theoretical guarantee can be used to justify the optimal complexity of the Multilevel Monte Carlo method. 
%
Numerical experiments are finally provided to verify the theoretical findings.

	
\end{abstract}

\section{Introduction}

Stochastic differential equations (SDEs) have found extensive applications in various disciplines such as finance, biology, chemistry, physics and engineering.
Since analytical solutions for most nonlinear SDEs are typically not accessible, there has been a growing interest in examining their numerical counterparts.
The past few decades have witnessed a lot of progress in this area, where the traditional setting imposed the global Lipschitz condition on the coefficient functions of SDEs \cite{platen1999introduction,milstein2004stochastic}.
A natural question then arises: what if the restrictive global Lipschitz condition is violated?
In 2011, Hutzenthaler, Jentzen and Kloeden \cite{hutzenthaler2011strong}  gave a negative answer to the question in the sense that 
the popularly used Euler-Maruyama method produces divergent numerical approximations when used to solve a large class of SDEs with super-linearly growing coefficients. 
Therefore, it is highly non-trivial to design and analyze numerical SDEs in the absence of the Lipschitz regularity of coefficients. 
Indeed, most nonlinear SDE models from computational finance not only have non-globally Lipschitz coefficients, but also have positive solutions, which
motivates the positivity-preserving numerical approximations. Following this direction, many researchers recently proposed and analyzed various positivity-preserving schemes for nonlinear SDEs with non-globally Lipschitz coefficients 
\cite{dereich2012euler,neuenkirch2014first,
Chassagneux_J._M.:financialSDEs_non-Lipschitz_SIAM16,
Mao:AitS._PositivityPreseved_BIT23,
Halidias_S.:AitS._positivity_DCDS23,
Higham_Mao_Szpruch:non-negativity_finance_DCDS13,
Alfonsi:CIR_SPL13,
Liu_Cao_Wang:AS_unconditionally_arkiv24,
Wang:mean-square_ACM23,
Wang:theta_mean-square_BIT20,
Lei_Gan_Chen:logarithmic_transformed_truncated_JCAM23,
Cai_Guo_Mao:Lotka–Volterra_Calcolo23,
Li_Cao:Lotka–Volterra_CNSNS23,
Mao_Wei_W.:Lotka–Volterra_JCAM21,
Chen_Gan_Wang:SIS_JCAM21,
Hong_Ji_Wang_Zhang:Lotka–Volterra_BIT22,Lord2024convergence,liu2023higher,emmanuel2021truncated,yang2021first,yang2023strong,
Yi_Hu_Zhao:positive_logarithmic2021}, to just mention a few.

As one of typical nonlinear SDEs with non-globally Lipschitz coefficients, the A\"{i}t-Sahalia interest rate model was widespreadly used in finance and economics, which was initially introduced by A\"{i}t Sahalia \cite{AitSahalia1996} and later expanded by \cite{szpruch2011numerical} to a general version, given by
\begin{equation}\label{eq:ait-sahalia_general[1AS24]}
     d X_t = ( \alpha_{-1} X_t^{-1} - \alpha_0 + \alpha_1 X_t - \alpha_2 X_t^r ) dt + \sigma X_t^{\rho} d W_t, \ t >0, \quad  X_0 = x_0 > 0.  
\end{equation}
Here $ \alpha_{-1}, \alpha_0, \alpha_1, \alpha_2 >0$, $r, \rho >1$ such that $r+1 \geq 2\rho$, and $\{W_t\}_{t \in [0, T]}$ is a standard Brownian motion.
As asserted by \cite{szpruch2011numerical}, the considered model \eqref{eq:ait-sahalia_general[1AS24]} 
is well-posed in the domain $(0, \infty)$ and admits a unique positive solution.
Clearly, the model has a polynomially growing drift that blows up at the origin and a diffusion term with superlinear growth.
These facts cause essential difficulties for the analysis of numerical approximations (see \cite{szpruch2011numerical} for more comments).
Due to the polynomially growing coefficients, the widely used Euler–Maruyama method, also known to be not positivity preserving, is apparently not a good candidate scheme for the model \eqref{eq:ait-sahalia_general[1AS24]}.
%

In \cite{szpruch2011numerical}, the authors discretized \eqref{eq:ait-sahalia_general[1AS24]} by using the backward Euler (BE) method and obtained positivity-preserving approximations. A strong convergence analysis was conducted there for the BE scheme applied to the model \eqref{eq:ait-sahalia_general[1AS24]} with $r+1 > 2 \rho$, but without any convergence rate disclosed. 
This gap was filled by \cite{Wang:theta_mean-square_BIT20}, where the authors achieved the mean-square convergence rate of order $1/2$ for stochastic theta methods applied to the A\"{i}t-Sahalia type model under the condition $r+1 \geq 2\rho$, also covering the general critical case $r+1 = 2\rho$. 
Recently, a kind of implicit Milstein method
was proposed for the A\"{i}t-Sahalia type model
in \cite{Wang:mean-square_ACM23}, where 
a mean-square convergence rate of order $1$ was successfully recovered.

Nevertheless, the implementation of implicit methods is computationally expensive as one needs to solve an implicit algebraic equation in each step.
In order to reduce computational costs, some researchers attempt to find explicit positivity-preserving  schemes.
Based on a Lamperti-type transformation, the authors of \cite{Halidias_S.:AitS._positivity_DCDS23} proposed an explicit, positivity-preserving  scheme with first-order strong convergence, for \eqref{eq:ait-sahalia_general[1AS24]} in the special critical case $r = 2$, $\rho = \tfrac{3}{2}$, which, however, only works for the special case.
The recent publication \cite{Mao:AitS._PositivityPreseved_BIT23} offered a positivity-preserving truncated Euler method for the non-critical case $r+1 > 2 \rho$, with a mean-square convergence rate of nearly $1/4$. 
More recently, the authors of \cite{Liu_Cao_Wang:AS_unconditionally_arkiv24} constructed an explicit, unconditionally positivity-preserving Euler-type method, which was proved to achieve a mean-square convergence of order $1/2$ for the general case $r+1 \geq 2\rho$.


In this paper, we aim to introduce a novel class of explicit Milstein-type schemes, which is easily implementable, unconditionally positivity-preserving and strongly convergent with order one.
On a uniform mesh within the interval $[0,T]$ with a time step-size $h = \frac{T}{N}, N \in \N$, we propose the following time stepping scheme:
\begin{equation}\label{eq:numer._in_intro.[1AS24]}
\left\{
\begin{array}{ll}
        Y_{n+1}
       =
       \Phi_h (Y_n) 
       +  (\alpha_{-1} Y_{n+1}^{-1} - \alpha_0 + \alpha_1 \Phi_h (Y_n) + f( \Phi_h (Y_n) )   ) h
       + g( \Phi_h (Y_n) ) \Delta W_n  \\
       \qquad \qquad \qquad \qquad
       + \tfrac{1}{2} ( |\Delta W_n|^2 - h  ) \hat{g}( \Phi_h (Y_n) ) ,  
       \ \
       n \in \{0,1,2,...,N-1\} , 
       \\
        Y_0  = x_0,
\end{array}
\right.
\end{equation}
where we denote $f(x) := - \alpha_2 x^r , g(x) := \sigma x^{\rho},  \hat{g}(x) := g'(x)g(x) = \rho \sigma^2 x^{2 \rho -1} , x \in (0,+\infty)$ and $\Delta W_n := W_{ t_{n+1} } - W_{ t_{n} }$.
Here, the crucial term $\Phi_h $ is a kind of corrective mapping depending on the time step-size $h$ and satisfying Assumption \ref{assump. of Phi[1AS24]}, which is incorporated to tackle the tough issue caused by the polynomially growing coefficients.
A typical choice of such operator could be a projection operator $\P_h$ defined by \eqref{eq:Ph-defn}.
%
In addition, the introduction of the implicit term $Y_{n+1}^{-1}$ is used to preserve the positivity of the original model. 
Such partial implicitness is, however, explicitly solved by finding a positive root of a quadratic equation (see \eqref{eq:Y-n-explicit-solution} below).

It is worthwhile to mention that, identifying a convergence rate of the proposed scheme for the considered model is highly non-trivial, due to a highly nonlinear drift that blows up at the origin, superlinearly growing diffusion coefficients and a mixture of implicitness and explicitness
in the drift part of the scheme.
Based on the globally monotone condition of $f$ and $g$ in $(0, \infty)$ (Lemma \ref{lemma:f_g_monotonicity[1AS24]}):
\begin{equation}
       ( x - y )( f(x) - f(y) ) + \tfrac{\upsilon - 1}{2} | g(x) - g(y) |^2 \leq L |x - y|^2,
       \quad 
       \text{for some } \upsilon>2,
        \quad
       x, y >0,
    \end{equation}
we introduce a novel approach to analyze mean-square error bounds of the new schemes, which does not rely on a priori high-order moment bounds of the numerical approximations.
In both non-critical ($r+1 > 2 \rho$) and general critical ($r+1 = 2 \rho$) cases, the expected order-one mean-square convergence is successfully attained for the proposed scheme.
%
More accurately, for the non-critical case $r + 1 > 2 \rho$ or the general critical case $r + 1 = 2 \rho$ with $\tfrac{\alpha_2}{\sigma^2} \geq 4r + \tfrac12$,
the proposed scheme is proved to achieve first-order convergence in the following sense (see Theorem \ref{thm(main):Order_one[1AS24]}):
\begin{equation}
    \E \big[ | X_{ t_{n} } - Y_n |^2 \big] 
    \leq C 
    h^2,
    \quad
    \forall \ h = \tfrac{T}{N}>0, \, T > 0, \, N \in \N.
\end{equation}
As far as we know, this is the first article to propose and analyze an explicit, unconditionally positivity-preserving method of first-order convergence for the A\"{i}t-Sahalia type model \eqref{eq:ait-sahalia_general[1AS24]} in both the non-critical and the general critical cases.

{\color{black}Next we would like to give comments on the relation between the Euler type approach in \cite{Liu_Cao_Wang:AS_unconditionally_arkiv24} and the Milstein type approach in this paper.
Both works are based on implicit treatment of the term $\alpha_{-1} x^{-1}$, to ensure the positivity-preserving. In order to properly handle the polynomially growing coefficients, 
the previous paper \cite{Liu_Cao_Wang:AS_unconditionally_arkiv24}  relied on a taming strategy and a lot of efforts are spent to obtain a priori high-order moment bounds of the numerical approximations, which is
unavoidable in the error analysis of the tamed scheme.
Different from \cite{Liu_Cao_Wang:AS_unconditionally_arkiv24}, the present work introduces the mapping
$\Phi_h $ to treat the polynomially growing coefficients, which also simplifies the strong convergence analysis by not requiring a priori high-order moment bounds of the numerical approximations.}

The remainder of this article is organized as follows. The next section presents some preliminaries.
In Section \ref{section:scheme[1AS24]}, the numerical scheme and its properties are presented.
The mean-square convergence is analyzed in Section \ref{section:error analysis[1AS24]}, with the convergence rate obtained.
Numerical experiments
are provided to verify the previous theoretical findings in Section \ref{section:numerical experiments[1AS24]}.

\section{Preliminaries}\label{section:preliminaries[1AS24]}

Let $\N$ be the set of all positive integers and $C$ be a positive constant that is independent 
of the time step-size and may vary at different appearance. 
Denote the Euclidean norm in $\R$ by $| \cdot |$ and set $T \in (0, \infty)$.
Given a filtered probability space $(\Omega, \F, \{ \F_t \}_{t\in [0,T]}, \mathbb{P} )$, we use $\E$ to represent the expectation and $L^p(\Omega;\R), p > 0$ to represent the space of all $\R$-valued random variables $\eta$ satisfing $\E[| \eta |^p] < \infty $, equipped with the norm $\| \cdot \|_{ L^p(\Omega;\R) }$ defined by:
\begin{equation}
    \| \eta \|_{ L^p(\Omega;\R) } := ( \E[ | \eta |^p ] )^{ \frac{1}{p} }, \quad \forall \ \eta \in L^p(\Omega;\R), \ 
    p > 0.
\end{equation}
Let us consider the A\"{i}t-Sahalia type model of the following form:
\begin{equation}\label{eq:ait-sahalia[1AS24]}
      d X_t = ( \alpha_{-1} X_t^{-1} - \alpha_0 + \alpha_1 X_t + f( X_t ) ) dt + g( X_t ) d W_t, \ t > 0, \quad  X_0 = x_0 > 0,
\end{equation}
where for short we denote 
\begin{equation}\label{eq:f_and_g[1AS24]}
    f(x) := - \alpha_2 x^r , 
    \ \
    g(x) := \sigma x^{\rho}, 
    \quad
     x \in D := (0, \infty),
\end{equation}
with $\alpha_{-1}, \alpha_0, \alpha_1, \alpha_2, \sigma > 0$ and $r,\rho > 1$. 
The monotonicity condition for $f$ and $g$ is presented in the following lemma, whose proof can be found in \cite[Lemma 5.9, Lemma 5.12]{Wang:mean-square_ACM23}.
    \begin{lem}       
\label{lemma:f_g_monotonicity[1AS24]}
    Let $f$ and $g$ be defined by \eqref{eq:f_and_g[1AS24]}.
    If the parameters in the model \eqref{eq:ait-sahalia[1AS24]} satisfy one of the following conditions:
    \begin{enumerate}[(1)]
        \item $r+1 > 2\rho$,
        \item $r+1 = 2\rho$, $\tfrac{\alpha_2}{\sigma^2} 
         >
         \tfrac{1}{8} \big(r + 2 + \tfrac{1}{r} \big) $,
    \end{enumerate}
    then for all $x,y \in D$, there exist constants $\upsilon > 2$ and $L > 0$ such that
    \begin{equation} \label{eq:monotonicity[1AS24]}
       ( x - y )( f(x) - f(y) ) + \tfrac{\upsilon - 1}{2} | g(x) - g(y) |^2 \leq L |x - y|^2.
    \end{equation}
\end{lem}

The well-posedness of the A\"{i}t-Sahalia type model \eqref{eq:ait-sahalia[1AS24]} has been established in \cite[Theorem 2.1]{szpruch2011numerical}, also quoted as follows.
\begin{lem}\label{wellposed of exact solution[1AS24]}     
    For any given initial data $X_0 = x_0 >0$, there exists a unique positive $\{ \F_t\}_{t\in [0,T]} $-adapted global solution with continuous sample paths $\{ X_t \}_{t \geq 0}$ to \eqref{eq:ait-sahalia[1AS24]}.
\end{lem}

Next we revisit some lemmas that give moment bounds for the exact solutions of the A\"{i}t-Sahalia type model \eqref{eq:ait-sahalia[1AS24]}.
For the non-critical case, the next lemma is quoted from \cite[Lemma 2.1]{szpruch2011numerical}.
\begin{lem}\label{lem:moment_case1[1AS24]}
     Let $r+1 > 2\rho$. Let $\{ X_t \}_{ t\in [0,T] }$ be the solution of (\ref{eq:ait-sahalia[1AS24]}). For any $p_0 \geq 2$ it holds that 
 \begin{equation}
     \sup_{t\in [0, \infty)} \E [ |X_t|^{p_0} ] < \infty, \quad
     \sup_{t\in [0, \infty)} \E [ |X_t|^{-p_0} ] < \infty.
 \end{equation}
\end{lem}

For the general critical case, we quote from \cite[Lemma 4.6]{Wang:theta_mean-square_BIT20} the following lemma.
\begin{lem}\label{lem.bound.case2[1AS24]}
        Let $r+1 = 2\rho$. Let $\{ X_t \}_{ t\in [0,T] }$ be the solution of \eqref{eq:ait-sahalia[1AS24]}. For any $2 \leq p_1 \leq \frac{\sigma^2 + 2 \alpha_2 }{ \sigma^2 }$ and for any $p_2 \geq 2$, it holds that 
 \begin{equation}
     \sup_{t\in [0, \infty)} \E [ |X_t|^{p_1} ] < \infty, \quad
     \sup_{t\in [0, \infty)} \E [ |X_t|^{-p_2} ] < \infty.
 \end{equation}
\end{lem}

Equipped with these moment bounds, the H\"{o}lder continuity of the solutions can also be derived, as stated in the following lemmas, 
quoted directly from \cite{Wang:theta_mean-square_BIT20}.
\begin{lem}
\cite[Lemma 4.4]{Wang:theta_mean-square_BIT20}
\label{lem:Holder_continuous_k+1>2rho[1AS24]}
    Let $r + 1 > 2 \rho$.
    Then it holds that, for any $p \geq 1$ and $t,s \in[0,T]$,
    \begin{align}
        \Vert X_t - X_s \Vert_{L^{p}(\Omega; \R)}
        & \leq 
        C \vert t-s \vert^{\frac12},\\
        \Vert X_t^{-1} - X_s^{-1} \Vert_{L^{p}(\Omega; \R)}
        & \leq 
        C \vert t-s \vert^{\frac12}.
    \end{align}
\end{lem}

\begin{lem}\cite[Lemma 4.7]{Wang:theta_mean-square_BIT20}
\label{lem:Holder_continuous_k+1=2rho[1AS24]}
    Let $r + 1 = 2 \rho$.
    Then for any $t,s \in[0,T]$ it holds that
    \begin{align}
        \Vert X_t - X_s \Vert_{L^{q_1}(\Omega; \R)}
        & \leq 
        C \vert t-s \vert^{\frac12},
        \quad
        2 
        \leq 
        q_1 
        \leq 
        \tfrac{1}{r} 
        \big( \tfrac{2 \alpha_2}{\sigma^2} + 1 \big), \\
        \Vert X_t^{-1} - X_s^{-1} \Vert_{L^{q_2}(\Omega; \R)}
        & \leq 
        C \vert t-s \vert^{\frac12},
        \quad
        2 
        \leq 
        q_2
        <
        \tfrac{1}{r} 
        \big( \tfrac{2 \alpha_2}{\sigma^2} + 1 \big).
    \end{align}
\end{lem}
The aforementioned lemmas help deduce the following lemma (cf. \cite[(4.21), (4.24)]{Wang:theta_mean-square_BIT20}).
\begin{lem}\label{lem:f_L^2_bounds[1AS24]}
Let $r + 1 \geq 2 \rho$. If one of the following conditions holds:
\begin{enumerate}[(1)]
    \item $r + 1 > 2 \rho$,
    \item $r + 1 = 2 \rho$, $\tfrac{\alpha_2}{\sigma^2} > 2r - \tfrac32$,
\end{enumerate}
then for any $t,s \in[0,T]$ we have
\begin{equation}
\begin{aligned}
\| f (X_t)- f (X_s) \|_{L^{2}(\Omega; \R)}
\leq
C |t - s|^{\frac12},
\\
\| g (X_t)- g (X_s) \|_{L^{2}(\Omega; \R)}
\leq
C |t - s|^{\frac12},\\
\| \hat g (X_t)- \hat g (X_s) \|_{L^{2}(\Omega; \R)}
\leq
C |t - s|^{\frac12},
\end{aligned}
\end{equation}
where $f,g$ are defined as \eqref{eq:f_and_g[1AS24]} and
\begin{equation}\label{eq:hat(g)[1AS24]}
    \hat{g}(x) := g'(x)g(x) = \rho \sigma^2 x^{2 \rho -1} ,
    \quad
    x \in D.
\end{equation}
\end{lem}


\section{The proposed explicit Milstein-type scheme}\label{section:scheme[1AS24]}

To numerically approximate the model \eqref{eq:ait-sahalia[1AS24]}, we do a temperal discretization on a uniform mesh within the interval $[0,T]$, with 
a uniform step-size $h = \tfrac{T}{N}, N \in \N $ and grid points $t_k := kh, k \in \{ 
0,1, ..., N  \}$. 
We propose a numerical scheme as follows:
\begin{equation}\label{eq:numeri_method[1AS24]}
\begin{split}
       Y_{n+1}
      & =
       \Phi_h (Y_n) 
       +  (\alpha_{-1} Y_{n+1}^{-1} - \alpha_0 + \alpha_1 \Phi_h (Y_n) + f( \Phi_h (Y_n) )   ) h
       + g( \Phi_h (Y_n) ) \Delta W_n  \\
       & \qquad \qquad \qquad \qquad
       + \tfrac{1}{2} ( |\Delta W_n|^2 - h  )  \hat{g}( \Phi_h (Y_n) ) , \ \ n \in \{0, 1, 2, ..., N-1\},
\end{split}
\end{equation}
with $ Y_0 = x_0$, where $\Delta W_n := W_{ t_{n+1} } - W_{ t_{n} }$ is the increment of the Brownian motion and $f,g$, $\hat g$ are defined by \eqref{eq:f_and_g[1AS24]}, \eqref{eq:hat(g)[1AS24]}.
Furthermore, $\Phi_h: D \rightarrow D$ is a kind of corrective mapping depending on the time step-size $h$, required to satisfy the following assumptions.

\begin{assumption}\label{assump. of Phi[1AS24]}
    Let $f, g, \hat{g}$ be defined as 
    \eqref{eq:f_and_g[1AS24]}, \eqref{eq:hat(g)[1AS24]}.
    For all $x, y \in D$ and $h = \tfrac{T}{N}>0$, there exist 
    constants $L_1, L_2, L_3 \geq 0 $ such that
    the operator $\Phi_h: D \rightarrow D$ obeys
    \begin{align}
    \label{eq:assump. Phi<=x[1AS24]}
         | \Phi_h(x) | 
         & 
         \leq |x| ,
         \\
    \label{eq:assump._x-Phi[1AS24]}
         | x - \Phi_h(x) |
         &
         \leq
         L_1  h ( 1 + |x|^{2 r + 1} ) 
         (  1 
            \wedge
            h ( 1 + |x|^{2r} )
            ) , 
    \\
    \label{eq:assump.Phi_x-Phi_y[1AS24]}
         | \Phi_h(x) - \Phi_h(y) |
         & 
         \leq ( 1 + L_2 h ) | x - y |,
         \\
    \label{eq:assump.f-g(Phi_x)-f-g(Phi_y)[1AS24]}
          | f( \Phi_h(x) ) - f ( \Phi_h(y) ) |^2
          + &
          | \hat{g}( \Phi_h(x)) - \hat{g}( \Phi_h(y) ) |^2
          \leq
          L_3  h^{-1} | x - y |^2 .
    \end{align}
\end{assumption}

The well-posedness and preservation of positivity are obvious, by noting that solving \eqref{eq:numeri_method[1AS24]} is nothing but finding a unique positive root of a quadratic equation, explicitly given by
\begin{equation}\label{eq:Y-n-explicit-solution}
\begin{split}
    Y_{n+1}
 = &
\frac12 \Bigg[ \Phi_h (Y_n)  +  \vartheta_n h +
    g(\Phi_h (Y_n) ) \Delta W_n 
    +
    \tfrac{1}{2} ( |\Delta W_n|^2 - h  )  \hat{g}( \Phi_h (Y_n) )
\\
&  + \sqrt{ 
\Big(
\Phi_h (Y_n)  
    + \vartheta_n h +
    g(\Phi_h (Y_n) ) \Delta W_n
    +
    \tfrac{1}{2} ( |\Delta W_n|^2 - h  )  \hat{g}( \Phi_h (Y_n) )
    \Big)^2 + 4 \alpha_{-1} h }
    \Bigg] >0,
\end{split}
\end{equation} 
where for short we denote 
\[
\vartheta_n :=  - \alpha_0 + \alpha_1 \Phi_h (Y_n) + f (\Phi_h (Y_n)).
\]
As a consequence, we have the following lemma.

\begin{lem}\label{lem:wellposed_of_Numerical method[1AS24]}     
     For any step-size $h = \tfrac{T}{N} > 0$,
     the numerical scheme \eqref{eq:numeri_method[1AS24]} is well-defined and positivity preserving,
     i.e., it admits a unique positive $\{ \F_{t_n}\}_{n=0}^N $-adapted solution $\{ Y_n \}_{ n=0 }^{N}, N \in \N $ for the scheme \eqref{eq:numeri_method[1AS24]} given a positive initial data. 
\end{lem}
In what follows we provide an {\color{black}example} operator $\Phi_h: D \rightarrow D$ fulfilling Assumption \ref{assump. of Phi[1AS24]}.

\begin{example}\label{example2[1AS24]}
Let $f, \hat{g}$ be denoted by \eqref{eq:f_and_g[1AS24]}, \eqref{eq:hat(g)[1AS24]} and let $r+1 \geq 2\rho$.
For any given $q \in [\tfrac{1}{2r}, \tfrac{1}{2r-2}]$ we define $\P_h: D \rightarrow D$ by
\begin{equation}\label{eq:Ph-defn}
    \P_h (x)
    :=
    \min \{ 1 , h^{ - q } |x|^{-1} \} x.
\end{equation}
\end{example}

Such a projection operator was introduced by \cite{Beyn_Isaak_Kruse:C-stability_B-consistency_JSC16,Beyn_Isaak_Kruse:C-stability_B-consistency_JSC17} to construct a projection Euler/Milstein schemes for SDEs in non-globally Lipschitz setting. Next we show that $\Phi_h = \P_h$ obeys all conditions in Assumption \ref{assump. of Phi[1AS24]}.
Firstly, one can easily confirm \eqref{eq:assump. Phi<=x[1AS24]} and \eqref{eq:assump.Phi_x-Phi_y[1AS24]} with $L_2 = 0$ by observing 
\begin{equation}\label{in example_for x-y[1AS24]}
    |\P_h(x)| \leq |x|, 
    \ \
    |\P_h(x) - \P_h(y) | \leq | x - y |, 
    \ \  
    \forall \ x,y \in D.
\end{equation}
Secondly, for all $x \geq h^{ - q} $, one has $\P_h(x) \leq h^{ - q} \leq x$. Thus for any $m > 0$ it holds that
\begin{equation}
    | x - \P_h(x) |
     \leq
       \1_{\{x \geq h^{ -q } \}} 2 | x |
       \leq
       2 h^m | x |^{ \frac{m}{q}  + 1 }  
       \leq
       2 h^m ( 1 + | x |^{ 2 m r + 1 })  ,
       \ \
       \forall \ x \in D,
\end{equation}
which confirms \eqref{eq:assump._x-Phi[1AS24]} by taking $L_1 = 2$ and $m=1$ or $m=2$.
Lastly,
since  $\P_h(x) \leq h^{ -q} $ for all $x \in D$, it is easy to obtain that for all $ x, y \in D$,
\begin{equation} \label{example. f(Phi_x)-f(Phy_y)[1AS24]}
    \begin{split}
        \big|  f( \P_h(x) ) - f( \P_h(y) )  \big|^2
        & 
        =
         \bigg|
              \int_0^1
                       f'( \theta \P_h(x) + ( 1 - \theta ) \P_h(y)  ) d\theta 
                        \cdot
                       ( \P_h(x) - \P_h(y) )
             \bigg|^2
             \\
        &
        \leq
        \alpha_2^2 r^2 
        \int_0^1 
        \big| \theta \P_h(x) + ( 1 - \theta ) \P_h(y)  \big|^{2r-2}  d\theta 
         \cdot
         \big| \P_h(x) - \P_h(y) \big|^2 
          \\
        & 
        \leq
        \alpha_2^2 r^2
        h^{ -q(2r-2) }   | x - y |^2 
         \\
        &
        \leq
        \alpha_2^2 r^2 
         h^{1 - q(2r-2)} \cdot h^{-1}  \vert x - y \vert^2
        \\
        & \leq
        \alpha_2^2 r^2
         T^{1 - q(2r-2)} \cdot h^{-1}  \vert x - y \vert^2,
    \end{split}
\end{equation}
where we used the fact that $1 - q(2r-2)  \geq  0$.
Similarly, we derive that for all $ x, y \in D$,
\begin{equation}\label{example. g(Phi_x)-g(Phy_y)[1AS24]}
    \begin{split}
        \big|  \hat{g}( \P_h(x)) - \hat{g}( \P_h(y) )  \big|^2
        & 
        \leq
        \int_0^1 
        | \hat{g}'( \theta \P_h(x) + ( 1 - \theta ) \P_h(y)  ) |^2  d\theta 
         \cdot
         |  \P_h(x) - \P_h(y) |^2 
     \\
        & 
        \leq
        \sigma^4 \rho^2 ( 2 \rho - 1 )^2 
        \int_0^1 
        | \theta \P_h(x) + ( 1 - \theta ) \P_h(y)  |^{4\rho-4}  d\theta 
         \cdot
         | \P_h(x) - \P_h(y) |^2 
         \\
        & \leq
        \sigma^4 \rho^2 ( 2 \rho - 1 )^2 
         T^{1 - q(2r-2)} \cdot h^{-1} \vert x - y \vert^2,
    \end{split}
\end{equation}
where we used the fact that $4\rho - 4 \leq 2r - 2$ in the last inequality. 
A combination of \eqref{example. f(Phi_x)-f(Phy_y)[1AS24]} and \eqref{example. g(Phi_x)-g(Phy_y)[1AS24]} confirms \eqref{eq:assump.f-g(Phi_x)-f-g(Phi_y)[1AS24]} with $L_3 = \big( \alpha_2^2 r^2  \vee \sigma^4 \rho^2 ( 2 \rho - 1 )^2 \big) T^{1 - q(2r-2)}$.

In light of the aforementioned evidence, the projection operator $\P_h$ satisfies all condtions in Assumption \ref{assump. of Phi[1AS24]}.




Armed with the above properties, we can now embark on the error analysis for the proposed scheme in the next section.



\section{Order one mean-square convergence}\label{section:error analysis[1AS24]}

In this section, we focus on the analysis of the mean-square convergence rate of the numerical scheme \eqref{eq:numeri_method[1AS24]}.
%
To begin with, we present the subsequent lemma regarding the error caused by the introduce of the corrective mapping $\Phi_h$ in $f,g$ and $\hat{g}$.

\begin{lem}\label{lemma:f_g(x)-f_g(Phi(x)) bounds[1AS24]}
    Let $f,g$ and $\hat{g}$ be defined by \eqref{eq:f_and_g[1AS24]} and \eqref{eq:hat(g)[1AS24]} satisfying $ r + 1 \geq 2 \rho$. Let Assumption \ref{assump. of Phi[1AS24]} hold. For all $x \in D$, there exists a constant $C$ independent of $h$ such that
\begin{equation}
    | f(x) - f( \Phi_h(x) )|
    \vee
    | g(x) - g( \Phi_h(x) )|
    \vee
    | \hat{g}(x) - \hat{g}( \Phi_h(x) )|
    \leq
    C h ( 1 + |x|^{ 3 r} ).
\end{equation}

\end{lem}

\begin{proof}
    By utilizing \eqref{eq:assump. Phi<=x[1AS24]} and \eqref{eq:assump._x-Phi[1AS24]}, it can be derived that
\begin{equation}\label{eq:f(x)-f(Phi(x))[1AS24]}
    \begin{split}
        | f(x) - f( \Phi_h(x) )|
        &
        =
        \bigg|  \int_0^1 f'(\theta x + ( 1 - \theta ) \Phi_h(x) ) d \theta 
        \cdot
        ( x - \Phi_h(x) ) \bigg| 
        \\
        &
        \leq
         \alpha_2 r \int_0^1 | \theta x + ( 1 - \theta ) \Phi_h(x) |^{r-1} d \theta \cdot | x - \Phi_h(x) |
        \\
        &                
        \leq
        \alpha_2 r |x|^{r-1} \cdot |  x - \Phi_h(x) |
        \\
        &
        \leq
        C h ( 1 + |x|^{ 3 r} ).
    \end{split}
\end{equation}
Noting that $r \geq 2\rho - 1 > \rho$, one can similarly deduce that
\begin{equation}\label{eq:g(x)-g(Phi(x))[1AS24]}
        | g(x) - g( \Phi_h(x) )|
        \leq
        \rho \sigma |x|^{\rho - 1} \cdot |  x - \Phi_h(x) |
        \leq
        C h ( 1 + |x|^{ 2 r + \rho } )
        \leq
        C h ( 1 + |x|^{ 3 r } ),
\end{equation}
and
\begin{equation}\label{hatg(x)-hatg(Phi(x))[1AS24]}
        | \hat{g}(x) - \hat{g}( \Phi_h(x) )|
        \leq
        \rho ( 2\rho - 1 ) \sigma^2   |x|^{2\rho - 2} \cdot |  x - \Phi_h(x) |
        \leq
        C h  ( 1 + |x|^{  3 r } ).
\end{equation}
The desired assertion can be achieved by combining \eqref{eq:f(x)-f(Phi(x))[1AS24]}, \eqref{eq:g(x)-g(Phi(x))[1AS24]} with \eqref{hatg(x)-hatg(Phi(x))[1AS24]}.

\end{proof}

Noting that 
$$ \int_{t_n}^{t_{n+1}} \int_{t_n}^{s} d W_l dW_s 
= 
\tfrac{     1}{2} \big( |\Delta W_n|^2 - h  \big) ,  
\ \
 n \in \{0, 1, 2, ..., N-1\},$$
one can rewrite
\eqref{eq:ait-sahalia[1AS24]} as: 
\begin{equation}\label{eq:ait_sahalia_gridpointVersion[1AS24]}
\begin{split}
        X_{ t_{n+1} }
        = 
        \Phi_h( X_{ t_{n} }) 
      & + 
        ( \alpha_{-1} X_{ t_{n+1} }^{-1} - \alpha_0 + \alpha_1 \Phi_h( X_{ t_{n} } ) + f( \Phi_h( X_{ t_{n} } ) )   ) h
        + 
        g( \Phi_h(X_{ t_{n} }) ) \Delta W_n  \\
      & \quad 
        + 
        \tfrac{1}{2} ( |\Delta W_n|^2 - h  )  \hat{g}( \Phi_h(X_{ t_{n} }) )
        + 
        R_{n+1}, \quad \forall  \ n \in \{ 0,1,..., N-1 \},
\end{split}
\end{equation}
where we denote
\begin{equation}\label{R_i[1AS24]}
\begin{split}
    R_{n+1} 
    := 
  & 
    \int_{t_n}^{ t_{n+1} } 
    [ \alpha_{-1} X_{s}^{-1} - \alpha_{-1} X_{ t_{n+1} }^{-1} + \alpha_1 X_s -  \alpha_1  \Phi_h( X_{ t_{n} } ) + f(X_s) - f( \Phi_h( X_{ t_{n} } ) ) ] 
    ds \\
    & \quad 
    + 
    \int_{t_n}^{ t_{n+1} }
    [  g(X_s) - g( \Phi_h( X_{ t_{n} } ) )  
       - 
       \hat{g}( \Phi_h( X_{ t_{n} } ) ) ( W_s - W_{t_n} )
        ] d W_s
     +
     X_{ t_{n} } - \Phi_h( X_{ t_{n} } ) .
\end{split}
\end{equation}
We would like to highlight that the introduction of the remainder term $R_n$ plays a crucial role in obtaining the expected convergence rate.
First, we need to estimate $\| R_i \|_{L^2(\Omega;\R)}$ and $\| \E (R_i| \F_{t_{i-1}}) \|_{L^2(\Omega;\R)}, i\in \{1,2,...,N\}$.

\begin{lem}\label{lem. R_n estimates [1AS24]}
    Let $\{ X_t \}_{ t\in [0,T] }$ and $\{ Y_n \}_{n=0}^N $ be the solutions of \eqref{eq:ait-sahalia[1AS24]} and \eqref{eq:numeri_method[1AS24]}, respectively. Let Assumption \ref{assump. of Phi[1AS24]} hold. 
    If one of the following conditions stands:
\begin{enumerate}[(1)]
    \item $r + 1 > 2 \rho$,
    \item \label{enumerate:theorem_critical_case[1AS24]}
    $r + 1 = 2 \rho$, $\tfrac{\alpha_2}{\sigma^2} \geq 4r + \tfrac12$,
\end{enumerate}
    then there exists a uniform constant $C$ independent of $h$, such that for all $n \in \{0,1,...,N-1 \} $, $N \in \N$,
    \begin{equation}
         \| R_{n+1} \|_{L^2(\Omega;\R)}
        \leq
        C h^{\frac{3}{2}},
        \quad
        \| \E [ R_{n+1} | \F_{t_n} ] \|_{L^2(\Omega;\R)}
        \leq
        C h^2.
    \end{equation}
\end{lem}

\begin{proof}

By the Minkowski inequality, we split $\| R_i \|_{L^2(\Omega;R)},i\in \{1,2,...,N\}$ into three terms as
  \begin{equation}\label{eq:R_i(r+1>2rho)[1AS24]}
    \begin{split}
        \| R_i \|_{L^2(\Omega;\R)}
        & \leq
         \underbrace{
         \bigg\|   
                 \int_{t_{i-1}}^{t_i}  
                        [ \alpha_{-1} X_{s}^{-1} - \alpha_{-1}   X_{ t_{i} } ^{-1} + \alpha_1 X_s -  \alpha_1 \Phi_h(  X_{ t_{i-1} } ) + f(X_s) - f( \Phi_h(  X_{ t_{i-1} } ) ) ]  ds  
         \bigg\|_{L^2(\Omega;\R)} 
                                    }_{ = : I_1 }   \\
        & \qquad
         +   
         \underbrace{
         \bigg\|   
                  \int_{t_{i-1}}^{t_i} 
                     [ g(X_s) - g( \Phi_h(  X_{ t_{i-1} } ) )  - \hat{g}( \Phi_h(  X_{ t_{i-1} } ) ) ( W_s - W_{ t_{i-1} } ) ] dW_s  
         \bigg\|_{L^2(\Omega;\R)}   
                                       }_{ = : I_2 }  \\
        & \qquad
         +
         \underbrace{
         \|     X_{ t_{i-1} } - \Phi_h(  X_{ t_{i-1} } )    \|_{L^2(\Omega;\R)}    
                                           }_{ = : I_3 }     .
    \end{split}
    \end{equation}
For the term $I_1$, 
{\color{black}by adding and subtracting the terms 
$h \alpha_1 X_{t_{i-1}}$ and $h f(X_{t_{i-1}})$, and further} 
utilizing 
\eqref{eq:assump._x-Phi[1AS24]} and
Lemma \ref{lemma:f_g(x)-f_g(Phi(x)) bounds[1AS24]}
we derive that
\begin{equation}\label{eq:I_1[1AS24]}
\begin{split}
    I_1
    &
    \leq
    \int_{t_{i-1}}^{t_i}
    \|  \alpha_{-1} X_{s}^{-1} - \alpha_{-1} X_{ t_{i} } ^{-1} \|_{L^2(\Omega;\R)} 
    ds
    +
    \int_{t_{i-1}}^{t_i}
    \| \alpha_1 X_s - \alpha_1  X_{ t_{i-1} }  \|_{L^2(\Omega;\R)}
    ds
    \\
    & \qquad 
     +
    h  \| \alpha_1  X_{ t_{i-1} } -  \alpha_1 \Phi_h(  X_{ t_{i-1} } ) 
        \|_{L^2(\Omega;\R)}
    +
    \int_{t_{i-1}}^{t_i}
    \| f(X_s) - f( X_{ t_{i-1} } )  \|_{L^2(\Omega;\R)}
    ds  
    \\
    & \qquad 
    +
    h
    \|  f( X_{ t_{i-1} } ) - f( \Phi_h(  X_{ t_{i-1} } ) )
    \|_{L^2(\Omega;\R)}
    \\
    &   
    \leq
    \alpha_{-1}
    \int_{t_{i-1}}^{t_i}
     \|   X_{s}^{-1} - X_{ t_{i} } ^{-1} \|_{L^2(\Omega;\R)} 
    ds
    +
    \alpha_1 
    \int_{t_{i-1}}^{t_i}
    \|  X_s -  X_{ t_{i-1} }  \|_{L^2(\Omega;\R)}
    ds
        \\
    & \qquad 
    +
    \int_{t_{i-1}}^{t_i}
    \| f(X_s) - f( X_{ t_{i-1} } )  \|_{L^2(\Omega;\R)}
    ds  
    +
    C h^{ 2 }
    \Big(   1 + \sup_{ t \in [0,T] } \| X_t \|^{ 3 r }_{ L^{ 6 r }(\Omega;\R) }   \Big)
     \\
    & 
    \leq  
    C h^{\frac{3}{2}} 
    \Big(   1 +  h^{\frac{1}{2}} \cdot \sup_{ t \in [0,T] } \| X_t \|^{ 3 r }_{ L^{ 6 r }(\Omega;\R) }   \Big),
\end{split}
\end{equation}
where we used Lemmas \ref{lem:Holder_continuous_k+1>2rho[1AS24]}, \ref{lem:Holder_continuous_k+1=2rho[1AS24]} and \ref{lem:f_L^2_bounds[1AS24]} 
in the last inequality.
   
For the term $I_2$, to simplify the denotation, we denote
\begin{equation}
    F(x) := \alpha_{-1} x^{-1} - \alpha_0 + \alpha_1 x - \alpha_2 x^r, 
    \ \
    x \in D.
\end{equation}
{\color{black}
By adding and subtracting the terms 
$g(X_{t_{i-1}})$ and $\hat{g}( X_{ t_{i-1} } )( W_s - W_{ t_{i-1} } )$ 
to the integrand of $I_2$, 
applying It\^o formula to $ g(x) = \sigma x^\rho$
and
making use of the H\"older inequality as well as the It\^o isometry} 
we get
\begin{equation}\label{eq:I_2_pre[1AS24]}
    \begin{split}
        |I_{2}|^2
        & \leq
        2 \int_{t_{i-1}}^{t_i}
         \| g(X_s) - g( X_{ t_{i-1} } ) 
            -
            \hat{g}( X_{ t_{i-1} } )( W_s - W_{ t_{i-1} } ) \|^2_{L^2(\Omega;\R)}
         ds
         \\ & \quad 
         +
         2 \int_{t_{i-1}}^{t_i}
         \| g(  X_{ t_{i-1} } ) - g( \Phi_h( X_{ t_{i-1} }) )
            +
            ( \hat{g}( X_{ t_{i-1} }) - \hat{g}( \Phi_h( X_{ t_{i-1} }) ) ) ( W_s - W_{t_{i-1}} )
            \|^2_{L^2(\Omega;\R)} 
        ds      
        \\
        & =
        2 \int_{t_{i-1}}^{t_i}
         \bigg\| 
             \int_{t_{i-1}}^{s}
             \Big( g'(X_l) F(X_l) + \frac{1}{2} g''(X_l) g^2(X_l) \Big)
             d l
             +
             \int_{t_{i-1}}^{s}
             \Big( \hat{g}(X_l) - \hat{g}( X_{ t_{i-1} }) \Big)
             d W_l
         \bigg\|^2_{L^2(\Omega;\R)}
         ds
         \\ & \quad 
         +
         2 \int_{t_{i-1}}^{t_i}
         \| g(  X_{ t_{i-1} } ) - g( \Phi_h( X_{ t_{i-1} }) )
            +
            ( \hat{g}( X_{ t_{i-1} }) - \hat{g}( \Phi_h( X_{ t_{i-1} }) ) ) ( W_s - W_{t_{i-1}} )
            \|^2_{L^2(\Omega;\R)} 
        ds \\
         & \leq
         4 h 
           \int_{t_{i-1}}^{t_i}
           \int_{t_{i-1}}^{s}
              \|  g'(X_l) F(X_l) + \frac{1}{2} g''(X_l) g^2(X_l)  
                                                  \|^2_{L^2(\Omega;\R)}  dl ds  \\
             & \quad 
         + 
         4 
           \int_{t_{i-1}}^{t_i}
            \int_{t_{i-1}}^{s}
            \|  \hat{g}(X_l) - \hat{g}( X_{ t_{i-1} }) 
                                                   \|^2_{L^2(\Omega;\R)} dl ds  \\
            & \quad                                             
         +
         4 h \|g(  X_{ t_{i-1} } ) - g( \Phi_h( X_{ t_{i-1} }) )\|^2_{L^2(\Omega;\R)} 
         +
         4 h^2 \| \hat{g}( X_{ t_{i-1} }) - \hat{g}( \Phi_h( X_{ t_{i-1} }) ) \|^2_{L^2(\Omega;\R)} 
         \\
         & \leq
            C h^3
           \Big( 
               1
               +
               \1_{\{\rho < 2\}}
               \sup_{ t \in [0,T] } 
               \Vert X_t^{-1} \Vert_{ L^{ 4 - 2 \rho }(\Omega;\R) } ^{4 - 2 \rho}
               +
               \sup_{ t \in [0,T] } 
               \Vert X_t \Vert_{ L^{ 2(r + \rho - 1) }(\Omega;\R) } ^{2(r + \rho - 1)}
               +
               \sup_{ t \in [0,T] }
            \Vert X_t \Vert_{ L^{ 6 r }(\Omega;\R) } ^{ 6 r }
            \Big) ,
    \end{split}
\end{equation}
where we used Lemmas \ref{lem:f_L^2_bounds[1AS24]}, 
\ref{lemma:f_g(x)-f_g(Phi(x)) bounds[1AS24]}
in the last inequality.
Observing that $4 - 2 \rho < 2$ and $2(r + \rho - 1) < 6 r $ and using the Lyapunov inequality imply that
\begin{equation*}
    \1_{\{\rho < 2\}} \Vert X_t^{-1} \Vert_{ L^{ 4 - 2 \rho }(\Omega;\R) }
    \leq
    \Vert X_t^{-1} \Vert_{ L^{ 2 }(\Omega;\R) },
    \quad
    \Vert X_t \Vert_{ L^{ 2(r + \rho - 1) }(\Omega;\R) }
    \leq
    \Vert X_t \Vert_{ L^{ 6 r }(\Omega;\R) },
    \quad
    \forall \ t \in [0,T].
\end{equation*}
As a result, one can deduce from \eqref{eq:I_2_pre[1AS24]} that
\begin{equation}\label{eq:I_2[1AS24]}
\begin{aligned}
          |I_{2}|^2
          & \leq
          C h^3
           \Big( 
               1
               +
               \Vert X_t^{-1} \Vert_{ L^{ 2 }(\Omega;\R) }
               +
               \sup_{ t \in [0,T] } 
               \Vert X_t \Vert_{ L^{ 6 r }(\Omega;\R) } ^{6 r}
            \Big) 
           .
\end{aligned}
\end{equation}
In view of \eqref{eq:assump._x-Phi[1AS24]}, one infers
\begin{equation}\label{I_3[1AS24]}
    I_3 
    \leq
    C h^{2} 
    \Big( 1 
          +
          \sup_{ t \in [0,T] } \| X_t \|^{ 4r+1 }_{L^{ 8r+2 }(\Omega;\R)}
    \Big).
\end{equation}
The Lyapunov inequality together with the fact that $ 6 r < 8 r + 2 $ helps us obtain
\begin{equation*}
    \Vert X_t \Vert_{ L^{ 6r }(\Omega;\R) }
    \leq
    \Vert X_t \Vert_{ L^{ 8 r + 2 }(\Omega;\R) },
    \quad
    \forall \ t \in [0,T].
\end{equation*}
A combination of \eqref{eq:I_1[1AS24]}, \eqref{eq:I_2[1AS24]} with \eqref{I_3[1AS24]} yields
\begin{equation}\label{eq:R_i_L2_estimate[1AS24]}
    \| R_i \|_{L^2(\Omega;\R)} 
    \leq
      C h^{\frac{3}{2}}
      \Big( 
          1 
       +
         \Vert X_t^{-1} \Vert_{ L^{ 2 }(\Omega;\R) }
       +
         \sup_{ t \in [0,T] } \| X_t \|^{ 4r+1 }_{L^{ 8r+2 }(\Omega;\R)}
       \Big).
\end{equation}
Since $\tfrac{\alpha_2}{\sigma^2} \geq 4r + \tfrac12$ implies $8r + 2 \leq \tfrac{\sigma^2 + 2 \alpha_2}{\sigma^2}$, by Lemma \ref{lem:moment_case1[1AS24]} and Lemma \ref{lem.bound.case2[1AS24]} we arrive at 
    \begin{equation}
         \| R_{n+1} \|_{L^2(\Omega;\R)}
        \leq
        C h^{\frac{3}{2}}, 
        \quad
        \forall \ n \in \{0,1,...,N-1 \} , N \in \N,
    \end{equation}
for the non-critical case $r + 1 > 2 \rho$ and the general critical case $r + 1 = 2 \rho$ with $\tfrac{\alpha_2}{\sigma^2} \geq 4r + \tfrac12$.

Now we turn to the estimations for $\| \E (R_i| \F_{i-1}) \|_{L^2(\Omega;\R)}, i\in \{1,2,...,N\}$. Using basic properties of the conditional expectation, one has
\begin{equation}\label{eq:I_2 vanish[1AS24]}
        \E  
        \bigg(  
         \int_{t_{i-1}}^{t_i} 
             \big[ g(X_s) - g( \Phi_h(  X_{ t_{i-1} } ) ) 
             -
             \hat{g}( \Phi_h(  X_{ t_{i-1} } ) ) (W_s - W_{t_{i-1}})
         \big] dW_s 
         \Big| \F_{t_{i-1}} 
         \bigg)
        =
        0.
\end{equation}
Therefore, 
{\color{black}
by adding and subtracting the terms
$\alpha_1 X_{t_{i-1}}$
and
$f(X_{t_{i-1}})$
to the integrands of $R_i$,
and using the It\^o formula twice to $\alpha_{-1}x^{-1} $ and $ \alpha_1 x + f(x),  x \in D$,
noting the fact that the It\^o integrals vanish under the conditional expectation}
and applying the H\"older inequality, one can show that 
\begin{equation}\label{eq:Ri|F_i-1<Ch^2[1AS24]}
\begin{split}
    \| \E (R_i| \F_{i-1}) \|_{L^2(\Omega;\R)}
    &
    \leq
       \|     X_{ t_{i-1} } - \Phi_h(  X_{ t_{i-1} } )    \|_{L^2(\Omega;\R)}   
    +
       \int_{t_{i-1}}^{t_i}
       \bigg\|  \E\Big( \alpha_{-1} X_{s}^{-1} - \alpha_{-1} X_{ t_{i} } ^{-1} \big| \F_{t_{i-1}} \Big) 
       \bigg\|_{L^2(\Omega;\R)} 
       ds
    \\ 
    & \quad
    +
    \int_{t_{i-1}}^{t_i}
    \bigg\| \E\Big( 
        \alpha_1 X_s + f(X_s) 
        - 
        \alpha_1  X_{ t_{i-1} } 
        - 
        f( X_{ t_{i-1} } ) 
        \big| \F_{t_{i-1}} 
    \Big)  \bigg\|_{L^2(\Omega;\R)}
    ds 
    \\
    & \quad
    +
    \int_{t_{i-1}}^{t_i}
    \big\| \alpha_1  X_{ t_{i-1} } -  \alpha_1 \Phi_h(  X_{ t_{i-1} } )
        +
        f( X_{ t_{i-1} } ) - f( \Phi_h(  X_{ t_{i-1} } ) )
    \big\|_{L^2(\Omega;\R)}
    ds
        \\
    & 
    \leq
        ( 1 + \alpha_1 h ) \|   X_{ t_{i-1} } -  \Phi_h(  X_{ t_{i-1} } ) \|_{L^2(\Omega;\R)} 
     \\
     & \quad
     +
     \int_{t_{i-1}}^{t_i}  
     \int_{s}^{t_i} 
           \big\|  
               \alpha_{-1} X_l^{-2} F(X_l) - \alpha_{-1} X_l^{-3} g^2(X_l) 
            \big\|_{L^2(\Omega;\R)} dlds  
                                  \\
    & \quad 
    + 
     \int_{t_{i-1}}^{t_i}  
     \int_{t_{i-1}}^{s} 
          \big\| ( \alpha_1 + f'(X_l) )  F(X_l) 
                + 
                \tfrac{1}{2} f''(X_l) g^2(X_l) 
                                            \big\|_{L^2(\Omega;\R)} dlds 
                                                \\ 
    & \quad 
     +
     h \big\| f(  X_{ t_{i-1} } ) - f( \Phi_h(  X_{ t_{i-1} } ) ) \big\|_{L^2(\Omega;\R)} 
     \\
     &\leq
         C h^2
         \Big( 1 
               +
             \sup_{ t \in [0,T] } \big\| X_t^{-3} \big\|_{ L^{2}(\Omega;\R) }
              +
             \sup_{ t \in [0,T] } \big\| X_t \big\|^{4r+1}_{ L^{8r+2}(\Omega;\R) }    \Big)   ,
\end{split}
\end{equation}
where we also used \eqref{eq:assump._x-Phi[1AS24]} and Lemma \ref{lemma:f_g(x)-f_g(Phi(x)) bounds[1AS24]} in the last inequality.
Similar to \eqref{eq:R_i_L2_estimate[1AS24]},
employing Lemma \ref{lem:moment_case1[1AS24]} for the non-critical case and Lemma \ref{lem.bound.case2[1AS24]} for the general critical case with $\tfrac{\alpha_2}{\sigma^2} \geq 4r + \tfrac12$ leads to
    \begin{equation}
        \| \E [ R_{n+1} | \F_n ] \|_{L^2(\Omega;\R)}
        \leq
        C h^2, 
        \quad
        \forall \ n \in \{0,1,...,N-1 \} , N \in \N.
    \end{equation}
The proof is thus completed.
\end{proof}

At the moment, we are well prepared to identify  the expected order-one mean-square convergence of the novel schemes.

\begin{thm}\label{thm(main):Order_one[1AS24]}
    Let $\{ X_t \}_{ t\in [0,T] }$ and $\{ Y_n \}_{n=0}^N $ be the solutions of \eqref{eq:ait-sahalia[1AS24]} and \eqref{eq:numeri_method[1AS24]}, respectively. Let Assumption 
    \ref{assump. of Phi[1AS24]} 
    hold. 
If one of the following conditions stands:
\begin{enumerate}[(1)]
    \item $r + 1 > 2 \rho$,
    \item $r + 1 = 2 \rho$, $\tfrac{\alpha_2}{\sigma^2} 
    \geq
    4r + \tfrac12$,
\end{enumerate}
    then there exists a uniform constant $C$ independent of $h$, such that for all $n \in \{1,2,...,N \} $, $N \in \N$,
\begin{equation}
    \| X_{ t_{n} } - Y_n \|_{L^2(\Omega;\R)} 
    \leq C 
    h.
\end{equation}
\end{thm}

\begin{proof}
For brevity, for all $k \in \{ 0,1,2, ..., N \}$ we denote
\begin{equation}\label{denotion:Delta_f[1AS24]}
\begin{split}
     e_k :=  X_{ t_{k} } - Y_k, 
     \quad 
     \Delta \Phi_{h,k}^{X, Y} := \Phi_h(  X_{ t_{k} } ) - \Phi_h( Y_k ), 
     \quad
   & \Delta f_k^{\Phi_h, X ,Y} := f( \Phi_h(  X_{ t_{k} } ) ) - f( \Phi_h(Y_k) ),
    \\
    \Delta g_k^{\Phi_h, X ,Y} := g( \Phi_h(  X_{ t_{k} } )) - g( \Phi_h(Y_k) ), \quad
    & 
    \Delta \hat{g}_k^{\Phi_h, X , Y} :=  \hat{g}( \Phi_h(  X_{ t_{k} } ) ) -  \hat{g}( \Phi_h(Y_k) ).    
\end{split}
\end{equation}
Bearing \eqref{eq:ait_sahalia_gridpointVersion[1AS24]} and  \eqref{denotion:Delta_f[1AS24]} in mind and subtracting \eqref{eq:numeri_method[1AS24]} from \eqref{eq:ait_sahalia_gridpointVersion[1AS24]}
infer that for all $n \in \{0, 1,2, ...,N-1\}$, 
\begin{equation}\label{eq:e_n[1AS24]}
\begin{split}
    e_{n+1} - h \cdot \alpha_{-1}  (  X_{ t_{n+1} }^{-1} - Y_{n+1}^{-1}  ) 
    & =
    ( 1 + h \alpha_{1} ) \Delta \Phi_{h,n}^{X, Y}
     +
     h \Delta f_n^{\Phi_h, X, Y}   
     +
     \Delta g_n^{\Phi_h, X, Y}  \Delta W_n
    \\
    & \qquad  + 
    \tfrac{1}{2}  ( |\Delta W_n|^2 - h  ) \Delta \hat{g}_n^{\Phi_h, X, Y} 
    +
    R_{n+1}.  
\end{split}
\end{equation}
%
Squaring both sides of \eqref{eq:e_n[1AS24]} yields that
\begin{equation}\label{eq:square_of_V_n+1,1[1AS24]}
\begin{split}
    & \vert e_{n+1} - h \cdot \alpha_{-1}  (  X_{ t_{n+1} }^{-1} - Y_{n+1}^{-1}  ) \vert^2 \\
    & =
    ( 1 + h \alpha_{1} )^2   | \Delta \Phi_{h,n}^{X, Y} |^2     
     +  
     h^2  |  \Delta f_n^{\Phi_h, X, Y}   |^2    
    +
     |  \Delta g_n^{\Phi_h, X, Y} \Delta W_n   |^2    
    +
     \tfrac{1}{4}  |  ( |\Delta W_n|^2 - h ) \Delta \hat{g}_n^{\Phi_h, X, Y}   |^2  \\   
    &\quad
    +  
    | R_{n+1} |^2   
          + 
          2 h( 1 + h \alpha_{1} ) \Delta \Phi_{h,n}^{X, Y}  \Delta f_n^{\Phi_h, X, Y}    
          +  
          2 ( 1 + h \alpha_{1} ) \Delta \Phi_{h,n}^{X, Y}  \Delta g_n^{\Phi_h, X, Y} \Delta W_n  \\   
          & \quad
          +  
           ( 1 + h \alpha_{1} ) \Delta \Phi_{h,n}^{X, Y} ( |\Delta W_n|^2 - h ) \Delta \hat{g}_n^{\Phi_h, X, Y}     
          +  
          2 ( 1 + h \alpha_{1} )  \Delta \Phi_{h,n}^{X, Y}  R_{n+1}  \\    
              & \quad
              + 
              2 h \Delta f_n^{\Phi_h, X, Y} \Delta g_n^{\Phi_h, X, Y} \Delta W_n     
              + 
              h \Delta f_n^{\Phi_h, X, Y} ( |\Delta W_n|^2 - h ) \Delta \hat{g}_n^{\Phi_h, X, Y}    
              + 
              2 h \Delta f_n^{\Phi_h, X, Y}  R_{n+1} \\   
                   & \quad
                   +  
                   \Delta g_n^{\Phi_h, X, Y} \Delta W_n   ( |\Delta W_n|^2 - h ) \Delta \hat{g}_n^{\Phi_h, X, Y}    
                   + 
                   2 \Delta g_n^{\Phi_h, X, Y} \Delta W_n  R_{n+1}  \\  
                         & \quad
                         +  ( |\Delta W_n|^2 - h ) \Delta \hat{g}_n^{\Phi_h, X, Y} R_{n+1}.  
\end{split}
\end{equation}
Before proceeding further, one first notes that
\begin{equation}
\begin{aligned}
    & \vert e_{n+1} - h \cdot \alpha_{-1}  (  X_{ t_{n+1} }^{-1} - Y_{n+1}^{-1}  ) \vert^2 \\
    & = 
    \vert e_{n+1} \vert^2
    -
    2 h \alpha_{-1} e_{n+1}  (  X_{ t_{n+1} }^{-1} - Y_{n+1}^{-1}  ) 
    +
    h^2 \alpha_{-1}^2 |  X_{ t_{n+1} }^{-1} - Y_{n+1}^{-1}  |^2\\
    & =
    \vert e_{n+1} \vert^2
    -
    2 h \alpha_{-1} | e_{n+1} |^2  \int_0^1
    (-1) \cdot  |  Y_{n+1} + \theta (X_{ t_{n+1} } - Y_{n+1})  |^{-2}
    d \theta
    +
    h^2 \alpha_{-1}^2 | X_{ t_{n+1} }^{-1} - Y_{n+1}^{-1}  |^2\\
    & \geq
    \vert e_{n+1} \vert^2.
\end{aligned}
\end{equation}
In addition, 
the facts
\begin{equation}
    \E \big[\vert \Delta W_n \vert^2\big] = h,
    \quad
    \E \big[ (\Delta W_n)^3 \big] = 0,
    \quad
    \E \big[ \vert \Delta W_n \vert^4 \big] = 3 h^2 ,
\end{equation}
help us deduce
\begin{equation*}
    \begin{split}
        & 
        \E \Big[ \big|\Delta g_n^{\Phi_h, X, Y} \Delta W_n \big|^2 \Big] 
        =
        \E \Big[ \big|\Delta g_n^{\Phi_h, X, Y}  \big|^2 \Big]
        \cdot
        \E \Big[ \big| \Delta W_n \big|^2 \Big]
        = 
        h \E \Big[ \big|\Delta g_n^{\Phi_h, X, Y} \big|^2 \Big] , 
        \\ 
        & \E\Big[ \big| ( |\Delta W_n|^2 - h ) \Delta \hat{g}_n^{\Phi_h, X, Y} \big|^2 \Big]
        =
        \E\Big[ \big| ( |\Delta W_n|^2 - h ) \big|^2 \Big]
        \cdot
        \E\Big[ \big| \Delta \hat{g}_n^{\Phi_h, X, Y} \big|^2 \Big]
         = 
         2 h^2 \E\Big[  \big|  \Delta \hat{g}_n^{\Phi_h, X, Y}  \big|^2 \Big],
    \end{split}
\end{equation*}
and
\begin{align*}
    &  
         \E \Big[ \Delta \Phi_{h,n}^{X, Y} \Delta g_n^{\Phi_h, X, Y} \Delta W_n \Big] 
         = 0 ,
         \ \
         \E \Big[  \Delta \Phi_{h,n}^{X, Y}  \big( |\Delta W_n|^2 - h \big) \Delta \hat{g}_n^{\Phi_h, X, Y}  \Big]
         = 0 , 
         \\
         &
         \E \Big[ \Delta f_n^{\Phi_h, X, Y} \Delta g_n^{\Phi_h, X, Y} \Delta W_n \Big] 
         = 0, 
             \ \
         \E \Big[ \Delta f_n^{\Phi_h, X, Y} \big( |\Delta W_n|^2 - h \big) \Delta \hat{g}_n^{\Phi_h, X, Y} \Big] 
         = 0, \\
         &
          \E\Big[ \Delta g_n^{\Phi_h, X, Y} \Delta W_n   \big( |\Delta W_n|^2 - h \big) \Delta \hat{g}_n^{\Phi_h, X, Y}  \Big] = 0,
\end{align*}
where we noted the terms $\Delta f_n^{\Phi_h, X, Y}$, $\Delta g_n^{\Phi_h, X, Y}$, $\Delta 
\hat g_n^{\Phi_h, X, Y}$ and $\Delta \Phi_{h,n}^{X, Y}$ are $\mathcal{F}_{t_n}$-measurable by construction.
The Young inequality also implies that for all $n \in \{0,1,2,...,N-1 \}$,
    \begin{align}
        2 h^2 \alpha_{1} 
        \E\Big[ \Delta \Phi_{h,n}^{X, Y}  \Delta f_n^{\Phi_h, X, Y} \Big]
        & \leq 
        h^2 \alpha_1^2 \E\Big[ | \Delta \Phi_{h,n}^{X, Y} |^2 \Big]
          +
        h^2 \E\Big[ | \Delta f_n^{\Phi_h, X, Y} |^2 \Big] , 
        \\
        2 h \E\Big[ \Delta f_n^{\Phi_h, X, Y}  R_{n+1} \Big] 
        & \leq
        h^2 \E\Big[ | \Delta f_n^{\Phi_h, X, Y} |^2 \Big]
          +
        \E\Big[ | R_{n+1} |^2 \Big] ,
         \\
        \E\Big[ \big( |\Delta W_n|^2 - h \big) \Delta \hat{g}_n^{\Phi_h, X, Y} R_{n+1} \Big]
        &
        \leq
        \tfrac{1}{2} h^2  \E\Big[ | \Delta \hat{g}_n^{\Phi_h, X, Y} |^2 \Big]
          + 
        \E\Big[ |R_{n+1}|^2 \Big],
        \\
        2 \E\Big[ \Delta g_n^{\Phi_h, X, Y} \Delta W_n  R_{n+1} \Big]
        & \leq
        h ( \upsilon - 2 ) \E\Big[ | \Delta g_n^{\Phi_h, X, Y} |^2 \Big]
          +
        \tfrac{1}{ \upsilon - 2 } \E\Big[ |R_{n+1}|^2 \Big] ,
    \end{align}
with $\upsilon>2$ coming from Lemma \ref{lemma:f_g_monotonicity[1AS24]}.
Accordingly, by taking expectations on both sides of \eqref{eq:square_of_V_n+1,1[1AS24]} and utilizing the above estimates, one immediately arrives at
\begin{equation}\label{E(square of V_j+1)[1AS24]}
\begin{split}
     \E\big[ | e_{n+1} |^2 \big] 
    & \leq
     ( 1 + h \alpha_{1} )^2   
     \E\Big[ \big| \Delta \Phi_{h,n}^{X, Y} \big|^2 \Big]    
     +  
     h^2  \E\Big[ \big|  \Delta f_n^{\Phi_h, X, Y}   \big|^2 \Big]   
     +
     h \E\Big[ \big|\Delta g_n^{\Phi_h, X, Y} \big|^2 \Big] \\  
     & \quad
     +
     \tfrac{1}{2} h^2  \E\Big[ \big| \Delta \hat{g}_n^{\Phi_h, X, Y} \big|^2 \Big]     
     +  
    \E\Big[ \big| R_{n+1} \big|^2 \Big]  
          + 
          2 h( 1 + h \alpha_{1} ) \E\Big[ \Delta \Phi_{h,n}^{X, Y}  \Delta f_n^{\Phi_h, X, Y} \Big] \\   
          & \quad 
          +  
          2 ( 1 + h \alpha_{1} ) 
          \E\Big[ \Delta \Phi_{h,n}^{X, Y}  R_{n+1} \Big]    
              + 
              2 h \E\Big[ \Delta f_n^{\Phi_h, X, Y}  R_{n+1} \Big]    
                   + 
                   2 \E\Big[ \Delta g_n^{\Phi_h, X, Y} \Delta W_n  R_{n+1} \Big] \\ 
                         &  \quad + 
                         \E\Big[ \big( |\Delta W_n|^2 - h \big) \Delta \hat{g}_n^{\Phi_h, X, Y} R_{n+1} \Big] 
                         \\
    & \leq
    \big[( 1 + h \alpha_{1} )^2  + h^2 \alpha_1^2  \big] 
     \E\Big[ \big| \Delta \Phi_{h,n}^{X, Y} \big|^2 \Big]    
     +  
     3 h^2  \E\Big[ \big|  \Delta f_n^{\Phi_h, X, Y}   \big|^2 \Big]   
     +
     h(\upsilon-1) \E\Big[ \big|\Delta g_n^{\Phi_h, X, Y} \big|^2 \Big] \\  
     & \quad
     +
     h^2  \E\Big[ \big| \Delta \hat{g}_n^{\Phi_h, X, Y} \big|^2 \Big]     
     +  
    (3+\tfrac{1}{\upsilon-2})\E\Big[ \big| R_{n+1} \big|^2 \Big]  
    + 
    2 h \E\Big[ \Delta \Phi_{h,n}^{X, Y}  \Delta f_n^{\Phi_h, X, Y} \Big] \\   
    & \quad 
    +  
    2 ( 1 + h \alpha_{1} ) 
    \E\Big[ \Delta \Phi_{h,n}^{X, Y}  R_{n+1} \Big]  
    \\
    &  = 
    \big[( 1 + h \alpha_{1} )^2  + h^2 \alpha_1^2  \big]   
     \E\Big[ \big| \Delta \Phi_{h,n}^{X, Y} \big|^2 \Big]    
     +
     (3+\tfrac{1}{\upsilon-2}) \E\Big[ \big| R_{n+1} \big|^2 \Big]  
     \\
     & \quad
      +
     2 h \bigg(
         \tfrac{\upsilon - 1}{2} \E\Big[ \big|\Delta g_n^{\Phi_h, X, Y} \big|^2 \Big]  
         +
         \E\Big[ \Delta \Phi_{h,n}^{X, Y}  \Delta f_n^{\Phi_h, X, Y} \Big]  
     \bigg) \\
     & \quad
      +
     h^2  \bigg(
         \E\Big[ \big| \Delta \hat{g}_n^{\Phi_h, X, Y} \big|^2 \Big]     
         +
         3 \E\Big[ \big|  \Delta f_n^{\Phi_h, X, Y}   \big|^2 \Big]   
     \bigg) \\
    & \quad
    +  
    2 ( 1 + h \alpha_{1} ) 
    \E\Big[ \Delta \Phi_{h,n}^{X, Y}  R_{n+1} \Big].  
\end{split}
\end{equation}
Furthermore, we note that all conditions required by Lemma \ref{lemma:f_g_monotonicity[1AS24]} are fulfilled, as
\begin{equation}\label{critical_condition_implies_monotone_condition[1AS24]}
    \tfrac{\alpha_2}{\sigma^2} 
    \geq
    4 r + \tfrac12
    =
    \tfrac{r}{8} + \tfrac{31}{8}(r - 1) + \tfrac{35}{8}
    >
    \tfrac{1}{8}(r + 35)
    >
    \tfrac{1}{8}\big(r + 2 + \tfrac{1}{r}\big).
\end{equation}
Therefore, using Lemma \ref{lemma:f_g_monotonicity[1AS24]} and \eqref{eq:assump.f-g(Phi_x)-f-g(Phi_y)[1AS24]} shows
\begin{equation}\label{iterateion[1AS24]}
\begin{aligned}
    \E\big[ | e_{n+1} |^2 \big] 
& 
\leq
    \big[( 1 + h \alpha_{1} )^2 + h^2 \alpha_1^2 + 2 L h + 3 L_3 h \big]   
     \E\Big[ \big| \Delta \Phi_{h,n}^{X, Y} \big|^2 \Big]   
     +
     (3+\tfrac{1}{\upsilon-2}) \E\Big[ \big| R_{n+1} \big|^2 \Big] 
     \\
     & \quad
    +  
    2 ( 1 + h \alpha_{1} ) 
    \E\Big[ \Delta \Phi_{h,n}^{X, Y}  R_{n+1} \Big]
 \\
& =
    \big[( 1 + h \alpha_{1} )^2 + h^2 \alpha_1^2  + 2 L h + 3 L_3 h \big]   
     \E\Big[ \big| \Delta \Phi_{h,n}^{X, Y} \big|^2 \Big]   
     +
     (3+\tfrac{1}{\upsilon-2}) \E\Big[ \big| R_{n+1} \big|^2 \Big] 
     \\
     & \quad
    +  
    2 ( 1 + h \alpha_{1} ) 
    \E\Big[ 
        \Delta \Phi_{h,n}^{X, Y} 
        \cdot
        \E\big[ R_{n+1}| \mathcal{F}_{t_n} \big]  
    \Big],
\end{aligned}
\end{equation}
where we also used the property of the conditional expectation in the last equality.
Further, utilizing the Young inequality gives
\begin{equation}
\begin{aligned}
    \E\big[ | e_{n+1} |^2 \big] 
    & \leq
    \big[( 1 + h \alpha_{1} )^2 + h^2 \alpha_1^2 + 2 L h + 3 L_3 h \big]   
     \E\Big[ \big| \Delta \Phi_{h,n}^{X, Y} \big|^2 \Big]    
     +
     (3+\tfrac{1}{\upsilon-2}) \E\Big[ \big| R_{n+1} \big|^2 \Big]  
     \\
     & \quad
    +  
    h ( 1 + h \alpha_{1} ) 
    \E\Big[ 
        \big\vert \Delta \Phi_{h,n}^{X, Y} \big\vert^2 
    \Big]
    +
    \tfrac{1 + h \alpha_{1}}{h}
    \E\Big[ 
        \big\vert \E\big[ R_{n+1}| \mathcal{F}_{t_n} \big] \big\vert^2
    \Big]  
    \\
    & =
    \big[( 1 + h \alpha_{1} )^2 + h^2 \alpha_1^2 + 2 L h + 3 L_3 h + h(1 + h \alpha_1) \big]   
     \E\Big[ \big| \Delta \Phi_{h,n}^{X, Y} \big|^2 \Big]    
     \\
     & \quad
     +
     (3+\tfrac{1}{\upsilon-2}) \E\Big[ \big| R_{n+1} \big|^2 \Big]  
    +
    \tfrac{1 + h \alpha_{1}}{h}
    \E\Big[ 
        \big\vert \E\big[ R_{n+1}| \mathcal{F}_{t_n} \big] \big\vert^2
    \Big].  
\end{aligned}
\end{equation}
In view of 
{\color{black}\eqref{eq:assump.Phi_x-Phi_y[1AS24]} and Lemma \ref{lem. R_n estimates [1AS24]}}, 
we obtain
\begin{equation}
\begin{aligned}
    \E\big[ | e_{n+1} |^2 \big] 
    & \leq
    \big[( 1 + h \alpha_{1} )^2 + h^2 \alpha_1^2 + 2 L h + 3 L_3 h + h(1 + h \alpha_1) \big]   
     \E\big[ | e_n |^2 \big]    
     \\
     & \quad
     +
     (3+\tfrac{1}{\upsilon-2}) C h^3  
    +
    \tfrac{1 + h \alpha_{1}}{h}
    C h^4
    \\
    & \leq
    (1 + C h) \E\big[ | e_n |^2 \big]
    +
    C h^3.  
\end{aligned}
\end{equation}
By iteration and the observation of $e_0 = 0$, we finally arrive at
\begin{equation}
\begin{aligned}
    \E \big[ \vert e_{n+1} \vert^2 \big]
    & \leq
    (1+ C h)^{ n + 1} \E\big[ | e_0 |^2 \big]
    +
    C h^3 \sum_{i=0}^{ n } (1+ C h)^i\\
    & \leq
    e^{ C T } \E\big[ | e_0 |^2 \big]
    +
    C h^2 T e^{C T}
    \leq
    C h^2.
\end{aligned}
\end{equation}
The proof is thus completed.

\end{proof}


\section{Numerical experiments}\label{section:numerical experiments[1AS24]}

In this section, some numerical experiments are presented to verify the previous theoretical findings,
for the approximation of the A\"{i}t-Sahalia type model
\begin{equation}
\left\{
\begin{array}{ll}
      dX_t = (\alpha_{-1}X^{-1}_t - \alpha_0 + \alpha_1 X_t - \alpha_2 X^{r}_t)dt + \sigma X^{\rho}_t dW_t,\quad t \in [0,T],\ \ T=1, 
     \\
      X_0 = 0.5.
\end{array}
 \right.
\end{equation}
To be more specific, we conduct numerical experiments by implementing the following semi-implicit projected Milstein method (SIPMM):
\begin{equation}
\begin{split}
       Y_{n+1}
      & =
       \Phi_h (Y_n) 
       +  (\alpha_{-1} Y_{n+1}^{-1} - \alpha_0 + \alpha_1 \Phi_h (Y_n) - \alpha_2 (\Phi_h(Y_n) )^r   ) h
       + \sigma( \Phi_h (Y_n) )^\rho \Delta W_n  \\
       & \qquad \qquad \qquad \qquad
       + \tfrac{1}{2} \rho \sigma^2 ( \Phi_h (Y_n) )^{ 2 \rho - 1 } ( |\Delta W_n|^2 - h  )   ,
       \ \
       n \in \{0,1,2,...,N-1\} ,
\end{split}
\end{equation}
where we take
$
    \Phi_h (x) = \min \{ 1 , h^{ - \frac{1}{2r-2} } |x|^{-1} \} x,
    \:
    x \in D,
$
so that the mapping coincides with Example \ref{example2[1AS24]} with $q=\tfrac{1}{2r-2}$.
Three sets of parameters are carefully chosen to meet the required conditions required by Theorem \ref{thm(main):Order_one[1AS24]}.

~

\textbf{Example 1} (non-critical case $ r + 1 > 2 \rho$): $\alpha_{-1} = \tfrac{3}{2} , \alpha_0 = 2 , \alpha_1 = 1, \alpha_2 = 13, \sigma = 1, r = 4 , \rho = 1.5$;

~

\textbf{Example 2} (critical case $ r = 3,  \rho = 2$): $\alpha_{-1} = \tfrac{3}{2} , \alpha_0 = 2 , \alpha_1 = 1, \alpha_2 = 13, \sigma = 1, r = 3 , \rho = 2$;

~

\textbf{Example 3} (critical case $ r=2, \rho = 1.5$): 
$\alpha_{-1} = \tfrac{3}{2} , \alpha_0 = 2 , \alpha_1 = 1, \alpha_2 = 13, \sigma = 1, r = 2 , \rho = 1.5$.


\begin{figure*}
\centering
\includegraphics[scale=0.7]{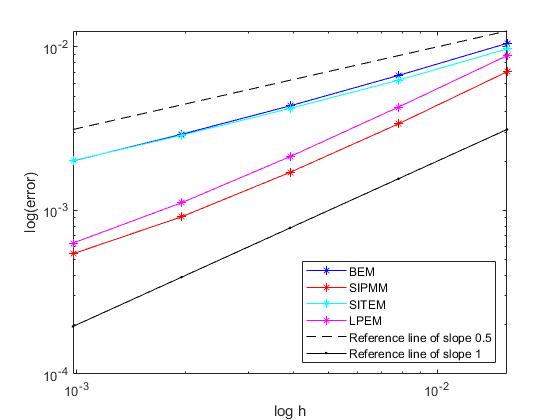}
\caption{Mean-square convergence rates for Example 1}  
\label{Fig.case1[1AS24]}
\end{figure*}

\begin{table}[htp]
\color{black}
	\centering
	\setlength{\tabcolsep}{8mm}
        \caption{\color{black}RMSE \& Time costs for BEM and SIPMM over $10^4$ Brownian paths in Example 1}\label{tab:time_costs[1AS24]}
	\begin{tabular}{c c c}
		\toprule[2pt]
		  &  BEM & SIPMM\\
		\midrule 
            $h=2^{-6}$ & err = 0.0105, time = 7.246s & err = 0.0070, time = 6.752s\\
		  $h=2^{-7}$ & err = 0.0067, time = 7.644s & err = 0.0034, time = 6.653s\\
            $h=2^{-8}$ & err = 0.0043, time = 9.386s & err = 0.0017, time = 6.910s\\
            $h=2^{-9}$ & err = 0.0029, time = 11.304s & err = 0.0009, time = 7.545s \\
            $h=2^{-10}$ & err = 0.0020, time = 13.897s & err = 0.0005, time = 8.114s\\
		\bottomrule [2pt]
	\end{tabular}
	\vspace{2pt}
\end{table}

\begin{figure*}
\centering
\includegraphics[scale=0.7]{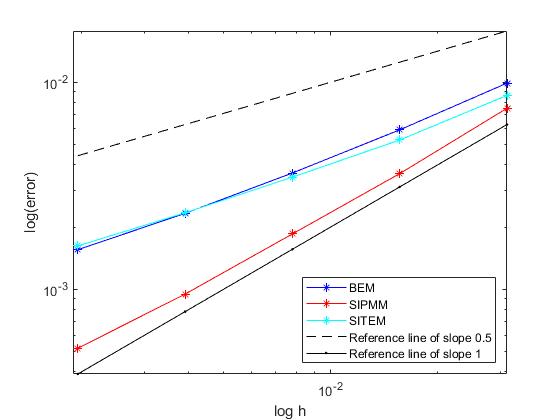}
\caption{Mean-square convergence rates for Example 2}  
\label{Fig.case2[1AS24]}
\end{figure*}

\begin{figure*}
\centering
\includegraphics[scale=0.7]{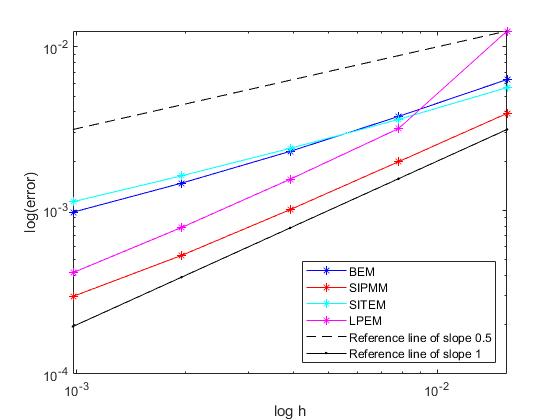}
\caption{Mean-square convergence rates for Example 3}  
\label{Fig.case3[1AS24]}
\end{figure*}

~

The backward Euler method (BEM) investigated in \cite{Wang:theta_mean-square_BIT20}, and the semi-implicit tamed Euler method (SITEM) proposed in \cite{Liu_Cao_Wang:AS_unconditionally_arkiv24} for the model \eqref{eq:ait-sahalia_general[1AS24]},
proved to be strongly convergent with order 0.5, are also implemented for comparison.
Numerical approximations produced by BEM with a fine step-size $h_{\text{exact}} = 2^{-15} $ are identified with ``exact" solutions, while various step-sizes $h=2^{-i},i=6,7,8,9,10$ are used for numerical approximations.
These two schemes together with our method are tested for the above three examples.
{\color{black}In addition, for the non-critical case in Example 1 and the critical case in Example 3, we conduct experiments using the numerical method proposed in \cite{Chassagneux_J._M.:financialSDEs_non-Lipschitz_SIAM16}, which combines a Lamperti-type transformation with a projected Euler-Maruyama method (LPEM).}
The expectation appearing in the mean-square error is approximated by calculating averages over $10000$ paths in the following numerical tests.


In {\color{black}Fig.\ref{Fig.case1[1AS24]}-Fig.\ref{Fig.case3[1AS24]}}, the mean-square convergence rates of {\color{black}BEM, SIPMM, SITEM and LPEM}
are depicted on a log-log scale.
There
one can easily see that the mean-square convergence rates of both BEM and SITEM are close to $0.5$, as opposed to a convergence rate close to $1$ for SIPMM. This can be detected more transparently from Table \ref{table1},
which confirms the theoretical results of Theorem \ref{thm(main):Order_one[1AS24]}. 
{\color{black}Besides,  a strong convergence of order 1 is observed for LPEM applied to Example 1 and Example 3.
Nevertheless, a mean square convergence rate of only order $0.5$ is achieved for the non-critical case 
$r + 1 > 2 \rho$ theoretically (see \cite[Corollary 4.6]{Chassagneux_J._M.:financialSDEs_non-Lipschitz_SIAM16} for more details).}

{\color{black}
Moreover, 
we present in Table \ref{tab:time_costs[1AS24]} 
the root mean square errors (RMSE) and time costs between BEM and SIPMM for Example 1 to illustrate the advantages of the proposed SIPMM. 
The numerical results clearly demonstrate that SIPMM exhibits significant reductions in both errors and time costs compared to BEM.
}

\begin{table}[htp]
	\centering
 \footnotesize
	\setlength{\tabcolsep}{7mm}
        \caption{A least square fit for the convergence rate $q$}\label{tab:least_squares_fit[1AS24]}
	\begin{tabular}{m{2cm}<{\centering} m{2.5cm}<{\centering} m{2.5cm}<{\centering} m{2.5cm}<{\centering} }
		\toprule[2pt]
		  & BEM & SIPMM & SITEM \\
		\midrule 
		 Example 1 
         & $q=0.5970$, $resid=0.0480$ 
         & $q=0.9282$, $resid=0.1300$
         & $q=0.5654$, $resid=0.0502$ 
         \\
		 Example 2 
         & $q=0.6264$, $resid=0.0580$ 
         & $q=0.9160$, $resid=0.0970$ 
         & $q=0.5657$, $resid=0.0422$ 
         \\
          Example 3
         & $q=0.6743$, $resid=0.0714$ 
         & $q=0.9338$, $resid=0.0668$ 
         & $q=0.5780$, $resid=0.0511$ 
         \\
		\bottomrule [2pt]
	\end{tabular}
	\vspace{2pt}
 \label{table1}
\end{table}



\section*{Declarations}

\subsection*{Ethical Approval} 
Not applicable

\subsection*{Availability of supporting data}
The authors confirm that the data supporting the findings of this study are available within the article.

\subsection*{Competing interests}
The authors declare no competing interests.

\subsection*{Funding}
This work is supported by the National Natural Science Foundation of China (Grant Nos. 12071488, 12371417, 11971488), the Natural Science Foundation of Hunan Province (Grant No.2020JJ2040).

\subsection*{Authors' contributions} 
All the authors contribute equally to this work.

\subsection*{Acknowledgments} 
Not applicable

\vskip6mm
\bibliography{ref}

\end{document}